\documentclass{article}

\usepackage[english]{babel}
\usepackage[letterpaper,top=2cm,bottom=2cm,left=3cm,right=3cm,marginparwidth=1.75cm]{geometry}

\usepackage{amsmath}
\usepackage{amsthm}
\usepackage{amssymb}
\usepackage{amsfonts}
\usepackage{graphicx}
\usepackage[colorlinks=true, allcolors=blue]{hyperref}
\usepackage{orcidlink}
\newcounter{dog}
\newtheorem{theorem}{Theorem}
\newtheorem{lemma}[theorem]{Lemma}
\newtheorem{question}[dog]{Question}

\title{Avoiding Monotone Arithmetic Progressions in Permutations of Integers}
\author{Sarosh Adenwalla\thanks{Department of Computer Science, University of Liverpool, UK, \texttt{Sarosh.Adenwalla@liverpool.ac.uk}, \orcidlink{0009-0009-8582-1281}}}

\begin{document}
\maketitle

\begin{abstract}
A permutation of the integers avoiding monotone arithmetic progressions of length $6$ was constructed in (Geneson, 2018). We improve on this by constructing a permutation of the integers avoiding monotone arithmetic progressions of length $5$. We also construct permutations of the integers and the positive integers that improve on previous upper and lower density results. In (Davis et al. 1977) they constructed a doubly infinite permutation of the positive integers that avoids monotone arithmetic progressions of length $4$. We construct a doubly infinite permutation of the integers avoiding monotone arithmetic progressions of length $5$. A permutation of the positive integers that avoided monotone arithmetic progressions of length $4$ with odd common difference was constructed in (LeSaulnier and Vijay, 2011). We generalise this result and show that for each $k\geq 1$, there exists a permutation of the positive integers that avoids monotone arithmetic progressions of length $4$ with common difference not divisible by $2^k$. In addition, we specify the structure of permutations of $[1,n]$ that avoid length $3$ monotone arithmetic progressions mod $n$ as defined in (Davis et al. 1977) and provide an explicit construction for a multiplicative result on permutations that avoid length $k$ monotone arithmetic progressions mod $n$. 
\end{abstract}

\section{Introduction}

A permutation, $a_1,a_2,a_3\ldots$, avoids monotone arithmetic progressions of length $k$ if there does not exist an increasing or decreasing subsequence of the permutation that forms a $k$-term arithmetic progression. We will refer to monotone arithmetic progressions of length $k$ as $k$-APs. Clearly, if a permutation avoids $k$-APs then it avoids $h$-APs for any $h>k$. In \cite{Davis77}, Davis et al. proved that the positive integers cannot be permuted to avoid $3$-APs and can be permuted to avoid $5$-APs and in \cite{Geneson18}, Geneson proved that the integers can be permuted to avoid $6$-APs. We prove that the integers can be permuted to avoid $5$-APs. These results leave open the questions of whether there exist permutations avoiding $4$-APs in the integers and the positive integers. Note that if there exists a permutation, $P$, of the integers avoiding $k$-APs, then the subsequence of positive integers in $P$ forms a permutation of the positive integers avoiding $k$-APs. Therefore, if there does not exist a permutation of the positive integers avoiding $k$-APs, there does not exist a permutation of the integers avoiding $k$-APs.  

We can work towards answering these open questions by considering the largest subset of the integers (or positive integers) that can be permuted to avoid $3$-APs or $4$-APs. Let $\alpha_{\mathbb{Z^+}}(k)$ be the supremum of $\limsup_{n\rightarrow \infty} \frac{|S\cap [1,n]|}{n}$ and $\beta_{\mathbb{Z^+}}(k)$ be the supremum of $\liminf_{n\rightarrow \infty} \frac{|S\cap [1,n]|}{n}$, over all sets of positive integers $S$ that can be permuted to avoid $k$-APs. These are the upper and lower density functions for the positive integers, as defined in \cite{LeSaulnier11}. We can define similar functions for the integers, whereby $\alpha_{\mathbb{Z}}(k)$ is the supremum of $\limsup_{n\rightarrow \infty} \frac{|R\cap [-n,n]|}{2n}$ and $\beta_{\mathbb{Z}}(k)$ is the supremum of $\liminf_{n\rightarrow \infty} \frac{|R\cap [-n,n]|}{2n}$ over all sets of integers $R$ that can be permuted to avoid $k$-APs. These are the upper and lower density functions for the integers which were first defined in \cite{Geneson18}. In \cite{LeSaulnier11} it was shown that $\alpha_{\mathbb{Z^+}}(4)=1$, $\alpha_{\mathbb{Z^+}}(3)\geq \frac{1}{2}$ and $\beta_{\mathbb{Z^+}}(3)\geq \frac{1}{4}$. These are currently the best known bounds and the authors of \cite{LeSaulnier11} conjectured that the latter two bounds are equalities. 

In \cite{Geneson18}, it was shown that $\alpha_{\mathbb{Z}}(k)=\beta_{\mathbb{Z}}(k)=1$ for all $k\geq 6$. By proving that the integers can be permuted to avoid $5$-APs, we show that $\alpha_{\mathbb{Z}}(5)=\beta_{\mathbb{Z}}(5)=1$. It was also proven in \cite{Geneson18} that $\alpha_{\mathbb{Z}}(3)\geq\frac{1}{2}, \beta_{\mathbb{Z}}(3)\geq\frac{1}{6}$ and $\beta_{\mathbb{Z^+}}(4)\geq\frac{1}{2}$. We improve on these latter two results to show that $\beta_{\mathbb{Z}}(3)\geq\frac{3}{10}$ and $\beta_{\mathbb{Z^+}}(4)=1$. We also show that $\alpha_{\mathbb{Z}}(4)=1$ and $\beta_{\mathbb{Z}}(4)\geq\frac{2}{3}$. All of these results have been proved by explicitly constructing permutations that avoid $k$-APs and attain those densities. There has not been any result found or technique developed towards upper bounding these values strictly below $1$.  

Another way of moving towards answering whether $4$-AP free permutations of the positive integers exist is to consider permutations with restrictions on the common difference of any $4$-APs that appear. In \cite{LeSaulnier11}, it was shown that the positive integers can be permuted to avoid $4$-APs with odd common difference (all $4$-APs have common difference divisible by $2$). We generalise this to show that, for any $k\geq 1$, the positive integers can be permuted to avoid $4$-APs with common difference not divisible by $2^k$ (equivalently, all $4$-APs in the permutation have common difference divisible by $2^k$). 

In \cite{Davis77}, it was shown that the positive integers can be arranged into a doubly infinite permutation, $\ldots a_{-2},a_{-1},a_0,a_1,a_2\ldots,$ that avoids $4$-APs but cannot be arranged in a doubly infinite permutation that avoids $3$-APs. Note that this implies that the integers cannot be arranged in a doubly infinite permutation that avoids $3$-APs. We will show that there exists a doubly infinite permutation of the integers that avoids $5$-APs. Whether there exists a doubly infinite permutation of the integers avoiding $4$-APs is still open. 

Let $a_1,a_2,\ldots,a_n$ be a permutation of $[1,n]$. Then a subsequence of this permutation, $a_{i_1},a_{i_2},\ldots,a_{i_k}$, is a $k$-AP mod $n$ if, for $d$ not equal to $0$  and an integer $a$, $$a_{i_t}\equiv a+td \pmod n,$$ for $1\leq t\leq k$. For example $4,2,5,1,3$ contains $3$-APs mod $5$ such as $4,2,5$ and $4,5,1$. Let $n$ be called \emph{$k$-permissible} if $[1,n]$ can be permuted to avoid $k$-APs mod $n$. In \cite{Nathanson77}, Nathanson shows that $n$ is $3$-permissible if and only if $n$ is a power of 2. Different proofs of this are given in \cite{Goh21,Karolyi17}. It is proven in both \cite{Davis77, Karolyi17} that $n$ is $5$-permissible for all $n$, which means all $n$ are $k$-permissible for all $k\geq 5$. It is not known which $n$ are $4$-permissible. We show that $n$ is $k$-permissible if and only if all the prime factors of $k$ are $k$-permissible, improving on a result in \cite{Nathanson77}. This allows us to prove that all $n$ that have no prime factors greater than $23$ are $4$-permissible. We also show that any permutation of [1,$2^k$] avoiding $3$-APs mod $2^k$ must obey a certain structure. From this we determine that there are $2^{2^k-1}$ such permutations, which was originally proven in \cite{Goh21} by a different method.  

The way we improve the results in these areas is by using modular arithmetic in creating these permutations. This is used in \cite{Davis77} where terms are arranged based on their parity in a doubly infinite permutation and in \cite{LeSaulnier11}. Any $k$-AP with terms that are distinct mod $n$ must form a $k$-AP mod $n$ with distinct terms. So, by grouping integers in a given interval together in the permutation according to their residues mod $2^k$ for some $k$ (having been rearranged to avoid $3$-APs), and ordering those residues in a way that avoids $3$-APs mod $2^k$, we can avoid $3$-APs appearing in certain ways in the permutation. This has been used in Theorems \ref{1}, \ref{4}, \ref{6} and \ref{7}. 

It is intuitive to think that constructing these permutations with large lower density is more difficult for the integers than for the positive integers. However, we can use the structure of the integers to achieve a permutation that avoids $3$-APs with a subset of the integers of greater lower density than the best known results for the positive integers. Let $A$ be a subset of the positive integers that can be permuted to avoid $3$-APs, of the form $[a_1,b_1],[c_1,d_1],[e_1,f_1],\ldots$ for $0<a_1<b_1<c_1<\ldots$ positive integers. Due to how these sets are usually permuted, in order for $A$ to avoid $3$-APs, the gaps between intervals are large compared to the size of the intervals. Now let $B$ be a subset of the integers that can be permuted to avoid $3$-APs such that for all $a,b\in B$, $a\neq -b$. We can choose the subset we use by initially picking an interval of positive integers, $[a_2,b_2]$, to be in our subset and then adding $[-d_2,-c_2]$ followed by $[e_2,f_2]$ and $[-h_2,-g_2]\ldots$ where $$a_2<b_2<c_2<d_2<e_2<f_2<g_2<h_2<\ldots,$$ are carefully chosen positive integers. In this way, we alternate between positive intervals and negative intervals. By doing this, the gap between $b_2$ and $e_2$ or between $-d_2$ and $-g_2$ can be large enough to avoid $3$-APs, and the gap between $a_2$ and $-c_2$ or between $e_2$ and $-h_2$ can be large enough to do so as well, while the gap between $b_2$ and $c_2$ or $f_2$ and $g_2$ can be smaller than the gaps between intervals in $A$. Let $B^+$ be the set containing the absolute value of all integers in $B$. Then $$\liminf_{n\rightarrow \infty} \frac{|B\cap [-n,n]|}{2n}=\liminf_{n\rightarrow \infty} \frac{|B^+\cap [1,n]|}{2n}=\frac{1}{2}\liminf_{n\rightarrow \infty} \frac{|B^+\cap [1,n]|}{n}.$$ Then $B^+$ is composed of $[a_1,b_2],[c_2,d_2],[e_2.f_2],[g_2,h_2],\ldots$ and so on. Due to the smaller gaps, the intervals in $B^+$ do not need to be spread out as much as in $A$. Thus the lower density of $B^+$ is not reduced as much and if constructed carefully, this can lead to such an improvement that despite the factor of $\frac{1}{2}$, $B$ has a larger lower density compared to $A$. This was used in Theorem \ref{8}.

\section{5-AP Free Permutations}

We will first construct a permutation of the integers, $P$, that avoids $5$-APs. We use that any finite subset of the integers can be permuted to avoid $3$-APs, as proven in \cite{Davis77}. If S is a permutation of a set of integers, $s_1,s_2,s_3,\ldots,s_n$, and $a, b$ are integers, then let $aS+b$ be $a\times s_1+b,a\times s_2+b,a\times s_3+b,\ldots,a\times s_n+b$, so the relative order of the terms is maintained. Note that $S$ avoids $k$-APs if and only if $aS+b$ avoids $k$-APs for non-zero $a$ as $s_{i_1},s_{i_2},\ldots,s_{i_k}$ is a $k$-AP if and only if $a\times s_{i_1}+b,a\times s_{i_2}+b,\ldots,a\times s_{i_k}+b$ is a $k$-AP. Also, if $T$ is another permutation of some set of integers, $t_1,t_2,t_3,\ldots,t_r$, then let $S,T$ denote the concatenation of the two permutations, $s_1,s_2,s_3,\ldots,s_n,t_1,t_2,t_3,\ldots,t_r$. Throughout the following constructions, our permutations are made up of intervals that have been permuted to avoid $3$-APs. We call such intervals \emph{blocks}.

\begin{theorem}\label{1} There exists a permutation of the integers, $P$, that avoids $5$-APs.\end{theorem}

\begin{proof} Let $X_i$ be a permutation of the integers in $[-8^i,-8^{i-1}-1]\cup[8^{i-1},8^i-1]$. We will construct $P$ using these blocks. Let \begin{align*}X_1&=-8,-4,4,-6,2,-2,6,-7,1,-3,5,-5,3,7,\\
X_{2n}&=8X_{2n-1}+7,8X_{2n-1}+3,8X_{2n-1}+5,8X_{2n-1}+1,8X_{2n-1}+6,8X_{2n-1}+2,8X_{2n-1}+4,8X_{2n-1}, \\
X_{2n+1}&=8X_{2n},8X_{2n}+4,8X_{2n}+2,8X_{2n}+6,8X_{2n}+1,8X_{2n}+5,8X_{2n}+3,8X_{2n}+7.
\end{align*}

Note that the integers in each block are grouped together by their residue classes mod $8$. In all odd indexed blocks, including $X_1$, the order of the residue classes is $0,4,2,6,1,5,3,7$ and in all even indexed blocks the order is $7,3,5,1,6,2,4,0$. 

Then let $$P=0,X_1,-1,X_2,X_3,X_4,X_5,\ldots$$. 

We will assume that $P$ contains a $5$-AP and derive a contradiction. Denote the terms of this progression as $a_1,a_2,a_3,a_4,a_5$, with common difference $d$. 

First we prove that each $X_i$ avoids $3$-APs by induction. Clearly $X_1$ avoids $3$-APs by inspection. Assume $X_k$ avoids $3$-APs for some $k$. Note that $X_{k+1}$ is made up of sub-blocks of the form $8X_k+a$ for $a\in\{0,1,2,3,4,5,6,7\}$. Then $8X_k+a$ avoids $3$-APs for any integer $a$. So any $3$-AP in $X_{k+1}$ cannot have all terms in one sub-block $8X_k+a$ for some $a\in\{0,1,2,3,4,5,6,7\}$. Assume there exists a $3$-AP in $X_{k+1}$, with terms $b_1,b_2,b_3$, which we shall prove cannot exist. As $b_1, b_2, b_3$ must belong to more than one sub-block, the common difference of the terms is not divisible by $8$. So $b_1$ and $b_2$ belong to different sub-blocks and $b_2$ and $b_3$ belong to different sub-blocks. Therefore, due to how $X_{k+1}$ is constructed, $b_1$ and $b_3$ belong to different sub-blocks as well. So $b_1,b_2,b_3,$ all have different residues mod $8$ and so $b_1 \pmod{8}$, $b_2 \pmod{8}$, $b_3 \pmod{8}$ form a $3$-AP mod $8$ with distinct terms. However, the order of the sub-blocks (and, therefore, the order of the residue classes mod $8$) in $X_{k+1}$ means that $X_{k+1}$ avoids $3$-APs mod $8$ with distinct terms. This can be seen by inspecting the permutations $$0,4,2,6,1,5,3,7 \qquad \text{ and }\qquad 7,3,5,1,6,2,4,0,$$ the orders of the residue classes mod $8$ in the blocks. Therefore, there is no such $b_1,b_2,b_3$ in $X_{k+1}$. So all $X_i$ avoid $3$-APs by induction. We split this into three cases based on whether $a_2$ and $a_3$ are in the same block, whether they are in different blocks or whether one of $a_2$ or $a_3$ is not in a block.  

\medskip
\noindent
$\textbf{Case 1}$: First we will consider the case where $a_2$ and $a_3$ are in different blocks. Denote the block $a_2$ belongs to as $X_i$. Then as $a_1$ appears before $a_2$ in $P$, we have that $|a_1|,|a_2|\leq8^i$ and $|d|<2\times8^i$. By the triangle inequality, $|a_3|=|a_2+d|<3\times8^i$. But as $a_3$ appears after $a_2$ in $P$ and is in a different block to $a_2$, it follows that $a_3$ is in $X_{i+1}$. Similarly $|a_4|<5\times8^i$ and $|a_5|<7\times8^i$ with both appearing after $a_3$ in $P$. So $a_4$ and $a_5$ are also in $X_{i+1}$. But then $a_3,a_4,a_5$ form a $3$-AP in $X_{i+1}$ which is a contradiction. \hfill $\diamond$

\medskip
\noindent
$\textbf{Case 2}$: Now we consider the case when $a_2$ and $a_3$ are in the same block, denoted $X_i$. As $a_2$ and $a_3$ are in $X_i$, this implies that $|a_2|<8^i$ and $|a_3|<8^i$ so $|d|<2\times8^i$. By the triangle inequality $|a_4|<3\times8^i$, and as $X_i$ avoids $3$-APs it follows that $a_4$ cannot be in $X_i$. Therefore, $a_4$ is in $X_{i+1}$. As $|a_5|<5\times8^i$ and $a_5$ appears in $P$ after $a_4$, we see that $a_5$ is in $X_{i+1}$. So $a_2,a_3,a_4,a_5$ are a $4$-AP with $a_2,a_3$ in $X_i$ and $a_4,a_5$ in $X_{i+1}$. We show that this cannot occur by proving that there can be no such $4$-AP over consecutive blocks $X_i$ and $X_{i+1}$. In fact this shows that there can be no $4$-AP over consecutive blocks at all as the only other way there could be is if there were either $3$ consecutive terms in $X_i$ or $3$ consecutive terms in $X_{i+1}$ which cannot happen. Let $b_1,b_2,b_3,b_4$ be a $4$-AP with common difference $d$ and with $b_1,b_2$ in $X_i$ and $b_3,b_4$ in $X_{i+1}$. We break this into sub-cases based on $d$. The first one is when $d$ is not divisible by $4$, the second is when $d$ is divisible by $4$ but not $8$ and the third is when $d$ is divisible by $8$.   

\medskip

$\textbf{Case 2a}$: If $d$ is not divisible by $4$, then $b_1,b_2,b_3,b_4$ are distinct mod $8$ and as they form a $4$-AP, we see that $b_1 \pmod{8}$, $b_2 \pmod{8}$, $b_3 \pmod{8}$, $b_4 \pmod{8}$ are a $4$-AP mod $8$ with distinct terms. However, the order of the residue classes mod 8 of $X_i,X_{i+1}$ is either $$7,3,5,1,6,2,4,0,4,2,6,1,4,3,7 \qquad\text{ or }\qquad 0,4,2,6,1,5,3,7,3,5,1,6,2,4,0,$$ depending on whether $i$ is even or odd. These sequences do not contain $4$-APs mod $8$ so there is no such $4$-AP with $d$ not divisible by $4$.\hfill$\diamond$

\medskip

$\textbf{Case 2b}$: If $d\equiv 4 \pmod{8}$, then $b_1\equiv b_3 \pmod{8}$ and $b_2\equiv b_4 \pmod{8}$ while $b_1$ is not equivalent to $b_2 \pmod{8}$. So there exists $x,y\in\{0,1,2,3,4,5,6,7\}$ such that $b_1\equiv b_3\equiv x \pmod{8}$ and $b_2\equiv b_4\equiv y \pmod{8}$. As $b_1$ appears in $X_i$ before $b_2$ then $x$ appears before $y$ in the order of the residue classes mod $8$ of $X_i$. As $b_3$ appears in $X_{i+1}$ before $b_4$ then $x$ appears before $y$ in the order of the residue classes mod $8$ of $X_{i+1}$. But by the definition of $X_i$ and $X_{i+1}$, if $x$ appears before $y$ in the order of the residue classes mod $8$ of $X_i$ then $y$ appears before $x$ in the order of the residue classes mod $8$ of $X_{i+1}$. Therefore, no such $4$-AP with $d\equiv 4 \pmod{8}$ exists. \hfill$\diamond$  

\medskip

$\textbf{Case 2c}$: Finally consider the case where $d\equiv 0 \pmod{8}$. Proceed by induction on $i$. First note that there is no $4$-AP with $b_1,b_2$ in $X_1$ and $b_3,b_4$ in $X_2$ with $d\equiv 0 \pmod{8}$. This can be seen as the only terms in $X_1$ with a difference divisible by $8$ are $$\{-4,4\}, \{-6,2\}, \{-2,6\}, \{-7,1\}, \{-3,5\}, \{-5,3\}.$$ So the only possible such progressions are $$\{-4,4,12,20\}, \{-6,2,10,18\}, \{-2,6,14,22\}, \{-7,1,9,17\}, \{-3,5,13,21\}, \{-5,3,11,19\}.$$ However, $2$ appears before $1$ in $X_1$, so $20$ appears before $12$ in $X_2$. Similarly, due to the order of terms in $X_1$ and the recursive definition of $X_2$, the third term in each of the above progressions appears after the fourth term in $X_2$. So there is no $4$-AP with the first two terms in $X_1$ and the last two terms in $X_2$ with common difference $d\equiv 0 \pmod{8}$.

Assume there is no $4$-AP with the first two terms in $X_{k-1}$ and the last two terms in $X_k$ with common difference $d\equiv 0 \pmod{8}$. Now consider a $4$-AP, $b_1,b_2,b_3,b_4$, with $b_1,b_2,$ in $X_k$ and $b_3,b_4$ in $X_{k+1}$ for $d\equiv 0 \pmod{8}$ and $k>1$. Then there exists $w\in\{0,1,2,3,4,5,6,7\}$ such that $$b_1\equiv b_2\equiv b_3\equiv b_4\equiv w \pmod{8},$$ so $b_1,b_2$ are in the sub-block $8X_{k-1}+w$ of $X_k$ and $b_3,b_4$ are in the sub-block $8X_k+w$ of $X_{k+1}$. Then let $c_j=b_j-w$ for $j\in\{1,2,3,4\}$. So $c_j$ is the multiple of $8$ that is less than or equal to $b_{j}$. As the smallest integer in $[8^{k-1},8^k-1]$ is a multiple of $8$ when $k>1$, and the smallest integer of $[-8^k,-8^{k-1}-1]$ is a multiple of 8, we see that $c_j$ is in the same block as $b_{j}$. So $c_1,c_2$ are in $X_k$ and $c_3,c_4$ are in $X_{k+1}$. As $c_1,c_2,c_3,c_4$ are multiples of $8$, it follows that $c_1,c_2$ are in the sub-block $8X_{k-1}$ and $c_3,c_4$ are in the sub-block $8X_k$. It is clear that $b_1$ appears before $b_2$ in $8X_{k-1}+w$, so $b_1-w=c_1$ appears before $b_2-w=c_2$ in $8X_{k-1}$. In the same way, as $b_3$ appears before $b_4$ in $8X_k+w$, we have that $b_3-w=c_3$ appears before $b_4-w=c_4$ in $8X_k$. Note that $8X_{k-1}$ is in $X_k$ and $8X_k$ is in $X_{k+1}$ so $8X_{k-1}$ appears before $8X_k$ in $P$. So as $b_1,b_2,b_3,b_4$ is a $4$-AP in $P$ with common difference $d$, the terms $b_1-w,b_2-w,b_3-w,b_4-w$ form a $4$-AP in $P$ with the same common difference $d$. Therefore, $c_1,c_2,c_3,c_4$ is a $4$-AP in $P$ with common difference $d$ with all terms $c_j$ and $d$ being divisible by $8$. 

As $c_1$ appears in the sub-block $8X_{k-1}$ before $c_2$ and $c_3$ appears in the sub-block $8X_k$ before $c_4$, it follows that $\frac{c_1}{8}$ appears before $\frac{c_2}{8}$ in $X_{k-1}$ and $\frac{c_3}{8}$ appears before $\frac{c_4}{8}$ in $X_k$. So $\frac{c_1}{8},\frac{c_2}{8},\frac{c_3}{8},\frac{c_4}{8}$ is a $4$-AP in $P$ with $\frac{c_1}{8},\frac{c_2}{8}$ in $X_{k-1}$ and $\frac{c_3}{8},\frac{c_4}{8}$ in $X_k$. Denote the common difference of this progression by $d'$ and note that $d'=\frac{d}{8}$. We have shown that if the common difference is not divisible by $8$, the progression cannot exist. So $d'\equiv 0 \pmod{8}$. However, this contradicts our inductive assumption. It follows that there is no $4$-AP with $d\equiv 0 \pmod{8}$ and with $b_1,b_2$ in $X_k$ and $b_3,b_4$ in $X_{k+1}$. 

This completes the induction and so, no such $4$-AP with $d$ divisible by $8$ appears in $P$. \hfill$\diamond$

\medskip
\noindent
Therefore, from the above sub-cases, no $4$-AP with terms $b_1,b_2$ in $X_i$ and $b_3,b_4$ in $X_{i+1}$ exists for any $i\geq 1$ in $P$. As mentioned, this means that there is no $4$-AP with terms in consecutive blocks. It follows that $a_2$ and $a_3$ cannot belong to the same block. This concludes the proof of Case $2$.\hfill$\diamond$

\medskip
\noindent
$\textbf{Case 3}$: The final case is when one of $a_2$ or $a_3$ is not in a block and is therefore equal to $0$ or $-1$. Neither $a_2$ nor $a_3$ can be $0$ as $0$ is the first term of $P$ and $a_1$ must appear in $P$ before $a_2$ and $a_3$. If $a_2$ is $-1$ then $a_1$ appears before $a_2$ in $P$ so $-8\leq a_1\leq 7$ and therefore $|d|\leq 8$. We then have that $|a_3|\leq 9$, also $|a_4|\leq 17$ and $|a_5|\leq 25$. As $a_3,a_4,a_5$ appear in $P$ after $-1$, they are all in $X_2$. But they would form a $3$-AP which is a contradiction as $X_2$ avoids these. So $a_2$ cannot be $-1$. If $a_3$ is $-1$ then $a_1,a_2,a_3$ is a $3$-AP contained in $0,X_1,-1$. As $a_3$ is odd, $a_1$ must be odd. By considering how $X_1$ is ordered, $a_2$ must be odd. So $d$ must be even which implies that the difference between $a_1$ and $a_3=-1$ is divisible by $4$. The only options for $a_1$ are then $-5, 3$ and $7$. Therefore, the only possible $3$-APs $a_1,a_2,a_3$ could be the following, $$\{-5,-3,-1\},\{3,1,-1\},\{7,3,-1\}.$$ However, $-3$ appears in $P$ before $-5$, similarly $1$ appears before $3$ and $3$ appears before $7$. This means $a_3$ cannot be $-1$. \hfill$\diamond$

\medskip
\noindent
From the above cases we can conclude that $P$ avoids $5$-APs.\end{proof}

This implies that $\alpha_{\mathbb{Z}}(5)=\beta_{\mathbb{Z}}(5)=1$. We now prove the same for doubly infinite permutations. Theorem \ref{2} follows easily from Theorem \ref{1} if we take $P$ to be the permutation defined above and consider a doubly infinite permutation $S$ where $2P+1$ is the permutation in $S$ that continues to the left and $2P$ is the permutation in $S$ that continues to the right. Below we provide an alternative permutation. 
\begin{theorem}\label{2} There exists a doubly infinite permutation of the integers, T, that avoids $5$-APs. \end{theorem}
\begin{proof}
Consider the doubly infinite permutation $$T=\dots,Z_{2i},Z_{2i-2},\ldots,Z_6,Z_4,Z_2,0,Z_1,-1,Z_3,Z_5,\ldots,Z_{2i-1},Z_{2i+1},\ldots,$$ where $Z_i$ is a permutation of the integers in $[-8^i,-8^{i-1}-1]\cup[8^{i-1},8^i-1]$ that avoids $3$-APs. Assume $T$ contains a $5$-AP $a_1,a_2,a_3,a_4,a_5$ with common difference $d$. Denote the block $a_2$ belongs to as $Z_j$. 

Note that if $a_3$ is in block $Z_i$, then either $i>j$, $i=j$ and $i$ is odd, $i<j$ or $i=j$ and $i$ is even. If $i>j$ then $i$ must be odd. This is because $a_3$ appears after $a_2$ in $T$ and, therefore, $Z_i$ appears to the right of $Z_j$ in $T$. So if $i>j$ then, from the structure of $T$, it follows that $i$ must be odd. The first two scenarios are covered in Case $1$, the third in Case $2$ and the fourth in Case $3$. Case $4$ is the final possibility that one of $a_2$ or $a_3$ is not in any blocks and is therefore either $0$ or $1$.

\medskip
\noindent
$\textbf{Case 1}$: Let $a_3$ be in $Z_i$ with $i\geq j$ and $i$ odd. Then $|a_2|,|a_3|\leq 8^i$ and $|d|< 2\times 8^i$. Both $a_4,a_5$ cannot be in $Z_i$ as $a_3,a_4,a_5$ would form a $3$-AP. So $a_5$ must be in a block that appears after $Z_i$ and therefore, as $i$ is odd, $|a_5|\geq 8^{i+1}$. But as $|a_5|=|a_3+2d|<5\times8^i$, we reach a contradiction. So $a_3$ is not in $Z_i$ for odd $i$ with $i\geq j$.\hfill $\diamond$

\medskip
\noindent
$\textbf{Case 2}$: Let $a_3$ be in $Z_i$ for $i<j$. As $a_3$ appears in $T$ after $a_2$, we see that $j$ is even. Assume $a_2$ is negative. The cases where $a_2$ is positive are similar. The first two sub-cases cover when $a_3<0$ and the third is when $a_3>0$.

\medskip

$\textbf{Case 2a}$: Let $a_3\in Z_i$ for $i<j-1$ with $a_3<0$. So $a_2\in[-8^j,-8^{j-1}-1]$ and $a_3\in[-8^i,-8^{i-1}-1]$. Therefore, $$8^{j-1}-8^i<d<8^j-8^{i-1},$$\text{  and it follows that } $$8^{j-1}-2\times8^i<a_4<8^j-2\times8^{i-1}.$$ It is clear that $a_4$ cannot be in $Z_j$ as $a_2$ is in $Z_j$ and $a_3\in Z_i$ appears between $a_2$ and $a_4$. So, $8^{j-1}-2\times8^i<a_4<8^{j-1}$ which means $a_4$ is in $Z_{j-1}$. Therefore, $$8^{j-1}-8^i<\frac{a_4-a_2}{2}=d<\frac{8^j+8^{j-1}}{2},$$ \text{ which means } $$8^{j-1}<2\times8^{j-1}-3\times8^i<a_5<\frac{1}{2}8^j+\frac{3}{2}8^{j-1}=\frac{11}{2}8^{j-1}.$$ So $a_5\in Z_j$ which is not possible as $a_2$ is in $Z_j$ and $a_3$ is not. \hfill $\diamond$

\medskip

$\textbf{Case 2b}$: Let $a_3\in Z_{j-1}$ and $a_3<0$ so $a_2\in[-8^j,-8^{j-1}-1]$ and $a_3\in[-8^{j-1},-8^{j-2}-1]$ which means $0<d<8^j-8^{j-2}$. Note that $j-1$ is odd and as $a_3,a_4,a_5$ cannot all be in $Z_{j-1}$, we see that $a_5$ appears after $Z_{j-1}$ in an odd indexed block. So $a_5\geq 8^j$ or $a_5\leq-8^j$. However, as $a_5>a_2$, it follows that $a_5\geq 8^j$. Therefore, $$\frac{8^j+8^{j-1}}{3}<\frac{a_5-a_2}{3}=d<8^j-8^{j-2},$$ \text{ and so }$$2\times8^{j-1}=\frac{-2\times8^{j-1}+8^j}{3}<a_4<8^j-2\times8^{j-2}.$$ This means $a_4$ is in $Z_j$ which is not possible as $a_2\in Z_j$ and $a_3$ is not.\hfill $\diamond$

\medskip

$\textbf{Case 2c}$: Let $a_3\in Z_i$ for $i<j$ and let $a_3>0$. Then $a_2\in[-8^j,-8^{j-1}-1]$, $a_3\in[8^{i-1},8^i-1]$ and $8^{i-1}+8^{j-1}<d$. So $a_1\in Z_k$ for $k\geq j$ and $k$ is even, as $a_1$ appears before $a_2$ in $T$. We have $a_1\in[-8^k,-8^{k-1}-1]$ as $a_1<a_2$. If $k>j$ then this is the same as Cases $2a$ and $2b$ but instead of considering $a_2$ and $a_3$ we consider $a_1$ and $a_2$. If $k=j$ then $d<8^j-8^{j-1}=7\times8^{j-1}$. Therefore $8^{i-1}\leq a_3<6\times 8^{j-1}$. It follows that $$8^{j-1}+2\times8^{i-1}<a_4<6\times8^{j-1}+7\times8^{j-1}=13\times8^{j-1}.$$ We see that $a_4$ cannot be in $Z_j$ because $a_2\in Z_j$ and $a_3$ is not, so $8^j\leq a_4<13\times 8^{j-1}$. Therefore, $$\frac{9}{2}8^{j-1}=\frac{8^j+8^{j-1}}{2}<\frac{a_4-a_2}{2}=d<7\times 8^{j-1},$$ and this means $$8^{j-1}=8^j-7\times8^{j-1}<a_4-d=a_3=a_2+d<6\times 8^{j-1},$$ so $a_3$ is in $Z_j$. This cannot happen as $a_3\in Z_i$ for $i<j$. \hfill $\diamond$

\medskip
\noindent
It follows from the above sub-cases that $a_3$ cannot be in $Z_i$ for $i<j$. This concludes the proof of Case 2. \hfill $\diamond$

\medskip
\noindent
$\textbf{Case 3}$: Let $j$ be even and $a_3\in Z_j$. Then $|d|<2\times8^j$. So $|a_1|\leq |a_2|+|d|<3\times8^j$. It is clear that $a_1$ cannot be in $Z_j$ as $a_1,a_2,a_3$ would form a $3$-AP in $Z_j$ so as $j$ is even and $a_1$ appears before $a_2$, we see that $|a_1|\geq 8^{j+1}$. This implies $8^{j+1}\leq |a_1|<3\times 8^j$ which is a contradiction.\hfill $\diamond$

\medskip
\noindent
$\textbf{Case 4}$: Finally consider the case where one of $a_2$ or $a_3$ is either $0$ or $-1$. If $a_2$ or $a_3$ were $0$ then $a_1,a_2,a_3$ or $a_2,a_3,a_4$ would be a $3$-AP in $T$ with $0$ as the middle term. This would involve $0$ appearing between a term and its additive inverse in $T$. So the only such $3$-APs in $T$ are $\{-8^i,0,8^i\}$ and $\{8^i,0,-8^i\}$ as all other terms belong to the same block as their additive inverse. So the only possible $5$-APs with $a_2$ or $a_3$ being $0$ are $$\{-2\times8^i, -8^i, 0, 8^i, 2\times8^i\}, \{2\times8^i, 8^i, 0, -8^i, -2\times8^i\}, \{8^i, 0, -8^i, -2\times8^i, -3\times8^i\}, \{-8^i, 0, 8^i, 2\times8^i, 3\times8^i\}.$$
However, $8^i$ and $-2\times 8^i$ belong to the same block so $0$ does not appear between them. Therefore the first three $5$-APs do not occur in $T$. Also $8^i, 2\times8^i,3\times8^i$ belong to $Z_{i+1}$ which avoids $3$-APs and therefore the fourth $5$-AP does not occur in $T$. So neither $a_2$ nor $a_3$ can be $0$.

Similarly, if $a_2$ or $a_3$ were $-1$, then $-1$ would be the middle term of a $3$-AP. The only such progressions are $\{-8^i-1,-1,8^i-1\}$ and $\{8^i-1,-1,-8^i-1\}$ as all other terms, $r$, belong to the same block as $-r-2$. So the only possible $5$-APs with $a_2$ or $a_3$ being $-1$ are $$\{2\times8^i-1, 8^i-1, -1, -8^i-1, -2\times8^i-1\}, \{-2\times8^i-1, -8^i-1, -1, 8^i-1, 2\times8^i-1\},$$ $$
\{-8^i-1, -1, 8^i-1, 2\times8^i-1, 3\times8^i-1\}, \{8^i-1, -1, -8^i-1, -2\times8^i-1, -3\times8^i-1\}.$$ 
However, $-8^i-1$ and $2\times8^i-1$ belong to the same block so $-1$ does not appear in $T$ between them. So the first three $5$-APs do not appear in $T$. Also $-8^i-1,-2\times8^i-1,-3\times8^i-1$ belong to $Z_{i+1}$ which avoids $3$-APs and therefore the fourth $5$-AP does not appear in $T$. \hfill $\diamond$

\medskip
\noindent
By the above cases, we can conclude that $T$ avoids $5$-APs. 
\end{proof}

\section{Monotone Arithmetic Progressions with restricted common difference}
We now prove that for any $k\geq 1$ there is a permutation of the positive integers that avoids $4$-APs with common difference not divisible by $2^k$. This generalises the result,

\begin{theorem} [\cite{LeSaulnier11}]
There exists a permutation of the positive integers in which every $4$-AP that occurs as a subsequence has common difference divisible by $2$.
\end{theorem} 
Their construction was based on separating integers into different blocks based on their residues modulo $2$, then ordering the blocks and scaling their intervals to have different lengths based on those residues. We appropriately apply the same techniques for residues modulo $2^k$ for $k\geq 1$ to achieve our result.

As $n$ is $3$-permissible if and only if $n$ is a power of $2$, it is equivalent to prove that for any $3$-permissible $n$, the positive integers can be permuted so that all $4$-APs have a common difference divisible by $n$. We will use the permutation in the theorem later in the paper.

\begin{theorem}\label{4} For every $n$ that is $3$-permissible, there exists a permutation of the positive integers such that all $4$-APs in the permutation have common difference divisible by $n$. \end{theorem}

\begin{proof} Choose $a\geq 3$. Let $X_{i}^j$ be a permutation of the integers equivalent to $j\pmod n$ in $[a^{i},a^{i+n})$ that avoids $3$-APs. We will refer to $X_{i}^j$ as a block. Note that a block only contains integers from one residue class mod $n$. Then let $S=s_1,s_2,\ldots,s_n$ be a permutation of the integers in $[1,n]$ such that $S$ avoids $3$-APs mod $n$. Now define $r(i)$ as the unique integer $r\in[1,n]$ such that $r\equiv i \pmod n$ and let $b_i=s_{r(i)}$, noting that $b_i=b_{i+n}$. Finally let $c_i=\lceil \frac{i}{n}\rceil n-r(i)+1$. Then 
$$R=X_{c_{-n+1}}^{b_1},X_{c_{-n+2}}^{b_2},\ldots,X_{c_{-n+i}}^{b_i},\ldots,X_{c_{-1}}^{b_{-1}},X_{c_0}^{b_0}, X_{c_1}^{b_1}, X_{c_2}^{b_2}, X_{c_3}^{b_3},\ldots, X_{c_i}^{b_i},\ldots, $$ which means
$$R=X_{0}^{s_1},X_{-1}^{s_2},X_{-2}^{s_3},\ldots,X_{-i}^{s_{i+1}},\ldots,X_{-n+2}^{s_{n-1}},X_{-n+1}^{s_n}, X_{n}^{s_1},X_{n-1}^{s_2},X_{n-2}^{s_3},\ldots,X_{n-i}^{s_{i+1}},\ldots,X_{2}^{s_{n-1}}, X_{1}^{s_{n}},$$ $$ X_{2n}^{s_1},X_{2n-1}^{s_2},\ldots,X_{n+2}^{s_{n-1}}, X_{n+1}^{s_n}, X_{3n}^{s_1}, \ldots,X_{tn}^{s_1},X_{tn-1}^{s_2},\ldots,X_{tn-i}^{s_{i+1}},\ldots,X_{(t-1)n+2}^{s_{n-1}},X_{(t-1)n+1}^{s_n}, X_{(t+1)n}^{s_1},\ldots.$$
This is a permutation of all positive integers as for all integers $v\geq 0$, there are exactly $n$ blocks that each contain all the integers equivalent to a distinct residue mod $n$ in the interval $[a^v,a^{v+1})$. We will show that any $4$-AP in $R$ has common difference $d\equiv 0\pmod n$. 

To prove this, we assume that $R$ contains a $4$-AP with terms $m_1,m_2,m_3,m_4,$ and common difference $d$ not equivalent to $0\pmod n$. We use this to obtain a contradiction. 

Note that if $X_i^{s_l}$ and $X_k^{s_{l'}}$ appear in $R$ then $i\equiv k\pmod n$ if and only if $s_{l}=s_{l'}$, and if $X_i^{s_l}$ appears in $R$ before $X_k^{s_l}$ then $i<k$ and all integers in $X_i^{s_l}$ are less than any integer in $X_k^{s_l}$. 

So take $d$ not equivalent to $0\pmod n$ and let $m_i\equiv s_{l_{i-1}} \pmod n$ for $i\in\{2,3,4\}$. Then let $m_2\in X_j^{s_{l_1}}$, $m_3\in X_h^{s_{l_2}}$ and $m_4\in X_k^{s_{l_3}}$. We will show that $l_1<l_2<l_3$ and use this to derive a contradiction. 

Letting $(t-1)n<j\leq t n$, we find $m_2\leq a^{j+n}$ and $d=m_2-m_1<m_2\leq a^{j+n}$ so $m_3<2\times a^{j+n}$ and $m_4<3\times a^{j+n}\leq a^{j+n+1}$. Therefore, as $m_3$ and $m_4$ appear after $m_2$ in $R$, it is clear that $$(t-1)n< j,h,k < j+n+1\text{ so }(t-1)n<j,h,k\leq j+n\leq (t+1)n.$$ The order of these blocks in $R$ is: 
$$X_{tn}^{s_1}, X_{tn-1}^{s_2}, X_{tn-2}^{s_3}, X_{tn-3}^{s_4},\ldots, X_{(t-1)n+2}^{s_{n-1}}, X_{(t-1)n+1}^{s_n}, X_{(t+1)n}^{s_1}, X_{(t+1)n-1}^{s_2}, X_{(t+1)n-2}^{s_3},\ldots, X_{tn+2}^{s_{n-1}}, X_{tn+1}^{s_n}.$$
We break this into two cases based on $h$.

\medskip
\noindent
\textbf{Case 1}: If $(t-1)n<h\leq tn$, then because $m_3\in X_h^{s_{l_2}}$ appears after $m_2\in X_j^{s_{l_1}}$ in $R$, we have $h\leq j$. If $h=j$ then $m_2$ and $m_3$ are in the same block so $m_3\equiv m_2\pmod n$ and $d\equiv 0\pmod n$ which is a contradiction. Therefore $(t-1)n<h<j\leq tn \text{ so } X_j^{s_{l_1}} \text{ and } X_h^{s_{l_2}} \text{ are of the form } X_{tn-i}^{s_{i+1}} \text{ for } 0\leq i\leq n-1,$ and it follows that $h<j$ implies $l_1<l_2$. Note that $$m_3\leq a^{h+n}\text{ and }d=m_3-m_2< a^{h+n},\text{ so } m_4=m_3+d<3\times a^{h+n}\leq a^{h+n+1}.$$ Then as $m_4\in X_k^{s_{l_3}}$ we have $k<h+n+1$, so it follows that $(t-1)n<k\leq h+n< (t+1)n$. We break this into two sub-cases based on $k$.

\medskip

\textbf{Case 1a}: If $(t-1)n<k\leq tn$, then because $m_4\in X_k^{s_{l_3}}$ appears after $m_3\in X_h^{s_{l_2}}$ in $R$, we have $k\leq h$. If $k=h$ then $m_3$ and $m_4$ are in the same block so $m_4\equiv m_3\pmod n$ and $d\equiv 0\pmod n$ which is a contradiction. Therefore 
\begin{flalign*} 
&& (t-1)n<k<h<tn \text{ and so }l_2<l_3. &&&  \hfill \diamond 
\end{flalign*}

\medskip

\textbf{Case 1b}: If $tn<k\leq h+n<(t+1)n$, then if $k=h+n$ we have $k\equiv h\pmod n$ so $s_{l_2} = s_{l_3}$. It follows that $m_4\equiv m_3\pmod n$ which implies $d\equiv 0\pmod n$ which is a contradiction. Therefore 
\begin{flalign*}
&& tn<k<h+n<(t+1)n \text{ and so }l_2<l_3. &&& \hfill \diamond
\end{flalign*}

\medskip
\noindent
Therefore if $(t-1)n<h<tn$, we have shown by the above sub-cases that $l_1<l_2<l_3$. This concludes Case 1. \hfill $\diamond$

\medskip
\noindent
\textbf{Case 2}: If $tn<h\leq j+n\leq (t+1)n$ then note that if $h=j+n$ it follows that $s_{l_1}=s_{l_2}$. So, $m_3\equiv m_2\pmod n$ and $d\equiv 0\pmod n$ which is a contradiction. Therefore, $$tn<h<j+n\leq (t+1)n\text{ and so }l_1<l_2.$$ As $tn<h<(t+1)n$ and $m_4\in X_k^{s_{l_3}}$ appears after $m_3\in X_h^{s_{l_2}}$ in $R$, it follows that either $tn<k\leq h$ or $k>(t+1)n$. But $k\leq (t+1)n$ and therefore $tn<k\leq h$. Note that if $k=h$ then $m_4\equiv m_3 \pmod n$ so $d\equiv 0\pmod n$ which is a contradiction. Therefore 
\begin{flalign*}
&& tn<k<h<(t+1)n \text{ so } l_2<l_3. &&& \hfill \diamond
\end{flalign*}

\medskip
\noindent
In either case we have that $l_1<l_2<l_3$. As $S$ is a permutation of $[1,n]$, we know $s_v\equiv s_w \pmod n$ if and only if $s_v=s_w$, which holds if and only if $v=w$. So $s_{l_1},s_{l_2},s_{l_3}$ are distinct mod $n$. $S$ avoids $3$-APs mod $n$ and $l_1<l_2<l_3$, so $s_{l_1},s_{l_2},s_{l_3}$ cannot form a $3$-AP mod $n$. This implies $m_2,m_3,m_4$ is not a $3$-AP and so $m_1,m_2,m_3,m_4$ is not a $4$-AP. 
Therefore there are no $4$-APs in $R$ with common difference not divisible by $n$.  \end{proof}

This contrasts with the result in \cite{Geneson18},
\begin{theorem}
[\cite{Geneson18}] For each integer $k > 1$, every permutation of the positive
integers contains an arithmetic progression of length $3$ with common difference not
divisible by $k$.
\end{theorem}
This is a generalisation of an analogous result in \cite{LeSaulnier11} for $k=2$. Note that if there existed an integer $k\geq 1$ and a permutation, $P$, such that $P$ contained no $3$-AP with common difference divisible by $k$, then the subsequence made up of the multiples of $k$ would contain no $3$-APs. Then, by dividing each term of this subsequence by $k$, we would produce a permutation of the positive integers that avoids $3$-APs. But no such permutation exists, as shown in \cite{Davis77}, and therefore for every integer $k\geq 1$, every permutation of the positive integers contains a $3$-AP with common difference divisible by $k$.

\section{Density Bounds}
We now prove the claimed density bounds. These improved lower bounds are given by alternating between blocks with positive and negative terms or by blocks with different residues modulo powers of $2$. This spreads out the terms and increases the lower density by reducing the size of the gaps required to avoid $3$-APs and $4$-APs. 

For the first result, we use the permutation from Theorem \ref{4}, but remove certain intervals to eliminate the remaining $4$-APs with common difference divisible by $2^k$.

\begin{theorem}\label{6} There exists a sequence of permutations of the positive integers such that $\beta_{\mathbb{Z^+}}(4)=1$.\end{theorem}
\begin{proof}  Choose $k\geq 1$ and let $n=2^k$. Let $a\geq 3$. Then let $R$ be the permutation as defined in Theorem \ref{4}. Define $R'$ to be a permutation with terms in the same relative order as $R$, but with the integers in $[\frac{a^{i+n}}{2},a^{i+n})$ removed for each block $X_i^j$ in $R$. As $R$ avoids $4$-APs with common difference not divisible by $n$, it follows that $R'$ avoids such $4$-APs as well. Assume $R'$ contains a $4$-AP with common difference, $d$, which must be divisible by $n$. Then denote this $4$-AP by $m_1, m_2, m_3, m_4$ and let the block $m_2$ belongs to be $X_i^j$, so $m_2<\frac{a^{i+n}}{2}$. Therefore, $$d=m_2-m_1<m_2<\frac{a^{i+n}}{2} \text{ and so } m_3=m_2+d<a^{i+n}.$$ As we have $d\equiv 0\mod n$, this means $m_1\equiv m_2\equiv m_3\equiv m_4 \equiv j\mod n$. As $m_3$ appears in $R'$ after $m_2$, we see that $m_3$ is either in $X_i^j$ or $m_3$ is in $X_f^j$ for $f\geq i+n$. In the latter case, $m_3\geq a^{i+n}$ which is a contradiction. So $m_3$ is in $X_i^j$. Similarly, $m_4=m_3+d<a^{i+n}$ and $m_4$ must be in $X_i^j$. However, $m_2, m_3, m_4$ form a $3$-AP in $X_i^j$ which is a contradiction as $X_i^j$ avoids $3$-APs. Therefore there are no $4$-APs in $R'$.  
As $\frac{|R'\cap[1,a^N]|}{a^N}$, is equal to \begin{align*} & \frac{a^n-2}{2n}\frac{\sum_{i=1}^{N-n}a^i}{a^N}+\frac{a^N-a^{N-1}}{na^N}+\frac{a^N-a^{N-2}}{na^N}+\ldots+\frac{a^N-a^{N-n+1}}{na^N} \\ &=\frac{a^n-2}{2n}\frac{a^{N-n+1}-a}{a^N(a-1)}+\frac{n-1}{n}-\frac{1}{n}\left(\frac{1}{a}+\frac{1}{a^2}+\ldots+\frac{1}{a^{n-1}}\right) \\ 
&=\frac{2n-1}{2n} + \frac{1}{2n(a-1)} + \frac{1}{na^{N-1}(a-1)}-\frac{1}{2na^{N-n-1}(a-1)}-\frac{1}{na^{n-1}(a-1)}-\frac{1}{n}\left(\frac{1-\frac{1}{a^{n-1}}}{a-1}\right)
\\ &=\frac{2n-1}{2n} + \frac{1}{na^{N-1}(a-1)}-\frac{1}{2na^{N-n-1}(a-1)}-\frac{1}{2n(a-1)}. \end{align*}
Therefore $$\liminf_{N\rightarrow \infty} \frac{|R'\cap[1,N]|}{N}= \frac{2n-1}{2n}-\frac{1}{2n(a-1)} = 1 - \frac{a}{2n(a-1)}.$$ 
As $k$ can be chosen to be arbitrarily large, $n=2^k$ can be arbitrarily large (while $a\geq 3$) so as $$\lim_{n\rightarrow \infty} 1 - \frac{a}{2n(a-1)} = 1,$$ we have $\beta_{\mathbb{Z}^+}(4)=1$.
\end{proof}

This leaves $\alpha_{\mathbb{Z}^+}(3)$ and $\beta_{\mathbb{Z}^+}(3)$ as the only functions for the positive integers to be resolved. 

\begin{theorem}\label{7} There exists a sequence of permutations of the integers such that $\beta_{\mathbb{Z}}(4)\geq\frac{2}{3}$ and $\alpha_{\mathbb{Z}}(4)=1$. \end{theorem}
\begin{proof} Consider the permutation $$T=X_2,Y_1,X_4,Y_3,\ldots,X_{2n},Y_{2n-1},X_{2n+2},Y_{2n+1},\ldots,$$ where $X_i$ is a permutation of the even integers in $[a^i, \frac{a^{i+2}}{3}-1]\cup[-\frac{a^{i+2}}{3}+1, -a^i]$ which avoid $3$-APs and $Y_i$ is a permutation of the odd integers in $[a^i, \frac{a^{i+2}}{3}-1]\cup[-\frac{a^{i+2}}{3}+1, -a^i]$ which avoid $3$-APs for $a>3$, with $a$ divisible by $3$. We will show $T$ avoids $4$-APs. Assume there exists a $4$-AP with terms $a_1,a_2,a_3,a_4,$ with common difference $d$. 

The first case covers when $a_1$ and $a_2$ are of the same parity and belong to the same block, the second case is when $a_1$ and $a_2$ are of the same parity and in different blocks, the third and fourth cases are when $a_1$ is even and $a_2$ is odd and the fifth case is when $a_1$ is odd and $a_2$ is even. So all cases are covered below.

\medskip
\noindent
$\textbf{Case 1}$: Let $a_1$ and $a_2$ belong to the same block. Without loss of generality, let this be $X_i$. Then $a_3$ cannot be in $X_i$, and $a_1,a_2,d,$ are even. It follows that $|d|<\frac{2}{3}a^{i+2}$ so $-a^{i+2}<a_3<a^{i+2}$. As $a_3$ is even and appears after $X_i$, it follows that $|a_3|\geq a^{i+2}$. So this case cannot occur. \hfill $\diamond$

\medskip
\noindent
$\textbf{Case 2}$: Let $a_1$ and $a_2$ be of the same parity in different blocks. Without loss of generality, let them be even. So let $a_1$ be in $X_i$ and $a_2$ be in $X_j$. Then $j>i+1$ and  $d$ is even. So $|d|<\frac{a^{j+2}}{3}+\frac{a^{i+2}}{3}$ and $|a_3|<\frac{2}{3}a^{j+2}+\frac{a^{i+2}}{3}<a^{j+2}$. As $a_3$ is even, if it appears after $X_j$ then $|a_3|\geq a^{j+2}$. Therefore $a_3$ is in $X_j$. It follows that $|a_4|<\frac{2}{3}a^{j+2}+\frac{a^{i+2}}{3}< a^{j+2}$ and similarly $a_4$ is in $X_j$. But then $a_2,a_3,a_4$ are all in $X_j$ which cannot occur as $X_j$ avoids $3$-APs. The case where $a_1$ is in $Y_i$ and $a_2$ is in $Y_j$ is similar. \hfill $\diamond$

\medskip
\noindent
$\textbf{Case 3}$: Let $a_1$ be in $X_i$ and $a_2$ be in $Y_j$ for $j>i$. Then $d$ is odd and $|d|<\frac{a^{j+2}}{3}+\frac{a^{i+2}}{3}$. Therefore $|a_3|<\frac{2}{3}a^{j+2}+\frac{a^{i+2}}{3}<a^{j+3}$. But as $a_3$ is even and occurs after $Y_j$, we have that $|a_3|\geq a^{j+3}$. So this case cannot occur. \hfill $\diamond$

\medskip
\noindent
$\textbf{Case 4}$: Let $a_1$ be in $X_i$ and $a_2$ be in $Y_{i-1}$. Then $d$ is odd and so $a_3$ is even and appears after $Y_{i-1}$, therefore $|a_3|\geq a^{i+2}$. We see that $$|d|=|a_2-a_1|<\frac{a^{i+1}}{3}+\frac{a^{i+2}}{3}\text{ and so }|a_3|=|a_2+d|<\frac{a^{i+2}}{3}+\frac{2}{3}a^{i+1}<a^{i+2}.$$ So this case cannot occur. \hfill $\diamond$

\medskip
\noindent
$\textbf{Case 5}$: Let $a_1$ be in $Y_i$ and $a_2$ be in $X_j$. Then $j>i+2$, and $d$ is odd with $|d|<\frac{a^{j+2}}{3}+\frac{a^{i+2}}{3}$. So $|a_3|<\frac{2}{3}a^{j+2}+\frac{a^{i+2}}{3}<a^{j+3}$ and as $a_3$ is odd and appears after $X_j$, this means $a_3$ is in $Y_{j-1}$ or $Y_{j+1}$.

\medskip

$\textbf{Case 5a}$: Let $a_3$ be in $Y_{j-1}$. Then $$|a_3|<\frac{a^{j+1}}{3}\text{ and so }|a_4|=|a_3+d|<\frac{a^{j+2}}{3}+\frac{a^{j+1}}{3}+\frac{a^{i+2}}{3}<a^{j+2}.$$ But as $a_4$ is even and appears after $Y_{j-1}$, we see that $|a_4|\geq a^{j+2}$. So this case cannot occur. \hfill $\diamond$

\medskip

$\textbf{Case 5b}$: Let $a_3$ be in $Y_{j+1}$. Then $$|a_3|<\frac{a^{j+3}}{3}\text{ so }|a_4|<\frac{a^{j+3}}{3}+\frac{a^{j+2}}{3}+\frac{a^{i+2}}{3}<a^{j+4}.$$ But as $a_4$ is even and appears after $Y_{j+1}$, we see that $|a_4|\geq a^{j+4}$. So this case cannot occur. \hfill $\diamond$

\medskip
\noindent
By the above sub-cases, we see that the case where $a_1$ is in $Y_i$ and $a_2$ is in $X_j$ cannot occur. This concludes Case 5.\hfill $\diamond$

\medskip
\noindent
The above cases show that $T$ avoids $4$-APs. 
$$\frac{|T\cap[-a^n,a^n]|}{2a^n}=\frac{a^2-3}{6}\frac{\sum_{i=1}^{n-2}a^i}{a^n}+\frac{a^n-a^{n-1}}{2a^n}=\frac{a^2-3}{6}\frac{a^{n-1}-a}{a^n(a-1)}+\frac{1}{2}-\frac{1}{2a}=\frac{a-3}{6(a-1)}-\frac{a^2-3}{6a^{n-1}(a-1)}+\frac{1}{2}.$$
It follows that $$\liminf_{N\rightarrow\infty}\frac{|T\cap[-N,N]|}{2N}=\frac{a-3}{6(a-1)}+\frac{1}{2}=\frac{2}{3}-\frac{1}{3(a-1)}.$$
As $a$ can be arbitrarily large, $\beta_{\mathbb{Z}}(4)\geq \frac{2}{3}$. 

\begin{align*}
\frac{|T\cap[-\frac{a^n}{3},\frac{a^n}{3}]|}{2\frac{a^n}{3}}&=\frac{a^2-3}{2}\frac{\sum_{i=1}^{n-2}a^i}{a^n}+3\frac{\frac{a^n}{3}-a^{n-1}}{2a^n}\\&=\frac{a^2-3}{2}\frac{a^{n-1}-a}{a^n(a-1)}+\frac{1}{2}-\frac{3}{2a}\\&=\frac{a-3}{2(a-1)}-\frac{a^2-3}{2a^{n-1}(a-1)}+\frac{1}{2}.\end{align*}
So $$\limsup_{N\rightarrow\infty}\frac{|T\cap[-N,N]|}{2N}=\frac{a-3}{2(a-1)}+\frac{1}{2}=1-\frac{1}{a-1}.$$
As $a$ can be arbitrarily large, $\alpha_{\mathbb{Z}}(4)=1$. 
\end{proof}

\begin{theorem}\label{8} There exists a sequence of permutations of the integers such that $\beta_{\mathbb{Z}}(3)\geq \frac{3}{10}$. \end{theorem}
\begin{proof} Consider the permutation $$K=Y_0,X_1,Y_2,X_3,Y_4,\ldots,X_{2k-1},Y_{2k},\ldots$$
where $X_i$ is a permutation of $[2^i,\lfloor\frac{8}{5}2^i\rfloor]$ which avoids $3$-APs and $Y_i$ is a permutation of $[-\lfloor\frac{8}{5}2^i\rfloor,-2^i]$ which avoids $3$-APs. We shall show that $K$ avoids $3$-APs. Let $a_1,a_2,a_3$ be a $3$-AP with common difference $d$ and let $a_1$ be in the block $X_i$. The cases where $a_1$ are in $Y_i$ are similar. Note that as $a_2$ appears in $K$ after $a_1$, it follows that $a_2$ must be in $X_i$ or $X_j$ for $j>i+1$ or $Y_k$ for $k>i$.

\medskip
\noindent
$\textbf{Case 1}$: If $a_2$ is in $X_i$ then $|d|<\frac{3}{5}2^i$. So $0<\frac{2}{5}2^i<a_3<\frac{11}{5}2^i<2^{i+2}$. But as $a_3$ cannot be in $X_i$ and so appears in $K$ after $X_i$, then $a_3\geq2^{i+2}$ or $a_3\leq-2^{i+1}<0$. \hfill $\diamond$

\medskip
\noindent
$\textbf{Case 2}$: If $a_2$ is in $X_j$ for $j>i+1$ then $2^j-\frac{8}{5}2^i<d<\frac{8}{5}2^j-2^i$ so $$ \frac{8}{5}2^j=2^{j+1}-\frac{8}{5}2^{j-2}\leq2^{j+1}-\frac{8}{5}2^i<a_3<\frac{16}{5}2^j-2^i<2^{j+2}.$$
But there are no terms in $K$ between $\frac{8}{5}2^j$ and $2^{j+2}$. \hfill $\diamond$

\medskip
\noindent
$\textbf{Case 3}$: If $a_2$ is in $Y_k$ for $k>i$ then $-\frac{8}{5}2^k-\frac{8}{5}2^i<d\leq-2^k-2^i$ so $$-2^{k+2}=-\frac{16}{5}2^k-\frac{8}{5}2^{k-1}\leq-\frac{16}{5}2^k-\frac{8}{5}2^i<a_3\leq-2^{k+1}-2^i<-\frac{8}{5}2^k.$$ 
But there are no terms in $K$ in this range. \hfill $\diamond$

\medskip
\noindent
By the above cases, $K$ avoids $3$-APs. By considering $$\lim_{N\rightarrow\infty} \frac{|K\cap[-2^N,2^N]|}{2\times2^N}=\lim_{N\rightarrow\infty} \frac{3}{5}\frac{\sum_{i=0}^{N-1}2^i}{2^{N+1}}=\frac{3}{5}\lim_{N\rightarrow\infty} \frac{2^N-1}{2^{N+1}}=\frac{3}{10},$$ we see that the lower density of $K$ is $\frac{3}{10}$ and so $\beta_{\mathbb{Z}}(3)\geq\frac{3}{10}$. 
\end{proof}
It is of interest to note that 
\begin{align*} \alpha_{\mathbb{Z}}(k)&=\sup_{R\subseteq\mathbb{Z}} \limsup_{n\rightarrow\infty} \frac{|R\cap[-n,n]|}{2n}\\&=\sup_{R\subseteq\mathbb{Z}} \limsup_{n\rightarrow\infty} \left(\frac{|R\cap[1,n]|}{2n}+\frac{|R\cap[-n,-1]|}{2n}\right) \\ &\leq \sup_{R\subseteq\mathbb{Z}} \limsup_{n\rightarrow\infty} \frac{|R\cap[1,n]|}{2n}+\sup_{R\subseteq\mathbb{Z}} \limsup_{n\rightarrow\infty} \frac{|R\cap[-n,-1]|}{2n} \\ &=\sup_{S\subseteq\mathbb{Z}^+} \limsup_{n\rightarrow\infty} \frac{|S\cap[1,n]|}{2n}+\sup_{S\subseteq\mathbb{Z}^+} \limsup_{n\rightarrow\infty} \frac{|S\cap[1,n]|}{2n}\\&=\sup_{S\subseteq\mathbb{Z}^+} \limsup_{n\rightarrow\infty} \frac{|S\cap[1,n]|}{n}\\&=\alpha_{\mathbb{Z^+}}(k),
\end{align*}
where the supremum is taken over all subsets $R$ of the integers (all subsets $S$ of the positive integers respectively) that avoid $k$-APs. Note that $$\alpha_{\mathbb{Z^+}}(k)=\alpha_{\mathbb{Z}}(k)=1 \text{ for all } k\geq 4 \text{ as } \alpha_{\mathbb{Z^+}}(4)=\alpha_{\mathbb{Z}}(4)=1$$ from \cite{LeSaulnier11} and Theorem \ref{7}. If the conjecture in \cite{LeSaulnier11} is true and $\alpha_{\mathbb{Z^+}}(3)=\frac{1}{2}$, then as $\frac{1}{2}\leq\alpha_{\mathbb{Z}}(3)\leq \alpha_{\mathbb{Z^+}}(3)$ by $\cite{Geneson18}$ and the above inequality, we would have $\alpha_{\mathbb{Z^+}}(3)=\alpha_{\mathbb{Z}}(3)=\frac{1}{2}$.

\section{Monotone Arithmetic Progressions Mod \texorpdfstring{$n$}{n}}
We begin by stating the below lemma from \cite{Karolyi17}. 
\begin{lemma}
[\cite{Karolyi17}] Let $G$ be a group and $N$ be a normal subgroup of $G$ such that both $N$ and $G/N$ have orderings (well orderings) avoiding $k$-APs. Then $G$ has an ordering (well ordering) avoiding $k$-APs.
\end{lemma}
We show how the above result applies in our context of well orderings in $G=\mathbb{Z}/n\mathbb{Z}$ and $N=\mathbb{Z}/m\mathbb{Z}$ where $m|n$. In \cite{Nathanson77}, Nathanson proved that if $b=p_1p_2\ldots p_t$ for $p_i<k$ not necessarily distinct and if $m$ is $k$-permissible then $mb$ is $k$-permissible. As it is clear that if $p_i<k$ then $p_i$ is $k$-permissible, the next result generalises this. 

\begin{lemma}\label{10} If $n$ and $m$ are $k$-permissible, then so is $mn$. 
\end{lemma}
\begin{proof} Let $a_1,\ldots,a_n$ be a permutation of $[1,n]$ that avoids $k$-APs mod $n$ and let $b_1,\ldots,b_m$ be a permutation of $[1,m]$ that avoids $k$-APs mod $m$. 

Then consider
$$S=m(a_1-1)+b_1, m(a_2-1)+b_1,\ldots, m(a_n-1)+b_1, m(a_1-1)+b_2, m(a_2-1)+b_2,\ldots, m(a_n-1)+b_2,$$ $$m(a_1-1)+b_3,\ldots, m(a_1-1)+b_m, m(a_2-1)+b_m,\ldots, m(a_n-1)+b_m.$$ We will show that $S$ is a permutation of $[1,mn]$ that avoids $k$-APs mod $mn$ and therefore show $mn$ to be $k$-permissible. First note that this is a permutation of $[1,mn]$ as $m(a_i-1)$ runs through all multiples of $m$ from $0\times m$ to $(n-1)\times m$ and $b_i$ runs through all integers from $1$ to $m$. Assume there exists a $k$-AP mod $mn$ in $S$, with terms $c_1,c_2,\ldots,c_k$ and with common difference $d$. 

\medskip
\noindent
$\textbf{Case 1}$: Assume $d$ is divisible by $m$ so there exists an integer $0<t<n$ such that $d=mt$. Then all terms of the progression are equivalent to each other mod $m$ so $$c_1\equiv c_2 \equiv \ldots \equiv c_k \equiv b_s \pmod{m}\text{ for some }s\in\{1,\ldots,m\}.$$ Then the $k$-AP is contained in $$m(a_1-1)+b_s, m(a_2-1)+b_s,\ldots, m(a_n-1)+b_s.$$ So $c_i=m(a_{r_i}-1)+b_s$ and $r_i$ is a strictly increasing sequence for $1\leq i\leq k$ due to how the terms are ordered in $S$. Then $$m(a_{r_i}-1)+b_s+d\equiv m(a_{r_{i+1}}-1)+b_s \pmod{mn}.$$ So $m(a_{r_i}-1)+d\equiv m(a_{r_{i+1}}-1) \pmod{mn}$ and therefore $a_{r_i}-1+t\equiv a_{r_{i+1}}-1 \pmod n$. This means $a_{r_i}+t\equiv a_{r_{i+1}} \pmod n$.

 Note that the terms $a_{r_i} \pmod n$ are all distinct. This is because if $a_{r_i}\equiv a_{r_j}\pmod n$ then as $a_{r_i},a_{r_j}\in[1,n]$, we have $a_{r_i}=a_{r_j}$ and so $c_i=c_j$. But the terms of the $k$-AP are all distinct, therefore $i=j$. It follows that the sequence $a_{r_1}, a_{r_2},\ldots,a_{r_k}$ forms a $k$-AP mod $n$ with common difference $t$ and is contained in $a_1, a_2,\ldots, a_n$ which is a contradiction. \hfill $\diamond$

\medskip
\noindent
$\textbf{Case 2}$: Assume $m$ does not divide $d$. Then every pair of consecutive terms in the progression, $c_i$ and $c_{i+1}$, have different residues mod $m$ and as terms in $S$ are arranged in blocks of their residues mod $m$, every term in the $k$-AP is distinct mod $m$. Note that the residues mod $m$ of the terms in $S$ are in the order $b_1,b_2,\ldots,b_m$. So $c_i=m(a_{f_{i}}-1)+b_{g_{i}}$ where $g_i$ is a strictly increasing sequence for $1\leq i\leq k$. It follows that $$m(a_{f_{i}}-1)+b_{g_{i}}+d\equiv m(a_{f_{i+1}}-1)+b_{g_{i+1}}\pmod{mn}.$$ Therefore $$m(a_{f_{i}}-1)+b_{g_{i}}+d\equiv m(a_{f_{i+1}}-1)+b_{g_{i+1}} \pmod m \text{ and so } b_{g_i}+d\equiv b_{g_{i+1}}\pmod m.$$ If $b_{g_i}\equiv b_{g_j}\pmod m$ then as $b_{g_i}, b_{g_j}\in[1,m]$, we have $b_{g_i}=b_{g_j}$. This means $c_i\equiv c_j\pmod m$ and so $i=j$. Therefore, the terms $b_{g_i}\pmod m$ are all distinct. It follows that the sequence $b_{g_1},b_{g_2},\ldots,b_{g_k}$ forms a $k$-AP mod $m$ with common difference $d$ and is contained in $b_1,\ldots,b_m$. This is a contradiction and so this case cannot occur.  \hfill $\diamond$

\medskip
\noindent
Therefore we can conclude that there is no $k$-AP mod $mn$ in $S$. 
\end{proof}
This implies that if all the primes dividing $n$ are $k$ permissible, so is $n$. Conversely, we know from Nathanson \cite{Nathanson77}, that if $m$ is not $k$ permissible and $m$ divides $n$ then $n$ is not $k$-permissible. One way to see this is that $m$ not being $k$-permissible means every permutation of the integers in $[1,m]$ contains a $k$-AP mod $m$ and when this progression is multiplied by the integer $\frac{n}{m}$, it becomes a $k$-AP mod $n$. For any permutation of $[1,n]$, the subsequence comprised of terms that are multiples of $\frac{n}{m}$ is a permutation of the integers in $[1,m]$ multiplied by $\frac{n}{m}$ so it contains a $k$-AP mod $n$.

Together this implies that $n$ is $k$-permissible if and only if all prime factors of $n$ are $k$-permissible. Therefore, we can restrict our attention to prime $n$. With this it is quite easy to rederive the result that the $n$ that are $3$-permissible are the powers of $2$ (see \cite{Goh21,Karolyi17,Nathanson77} for alternate proofs). Any permutation of $[1,2]$ clearly avoids $3$-APs mod $2$. For any odd prime $p$ and permutation $a_1,\ldots,a_p$ of $[1,p]$, the terms $a_1,a_2,2a_2-a_1\pmod p$ form a $3$-AP mod $p$. This is because $2a_2-a_1\pmod p$ must be distinct from $a_1$ and $a_2$ as otherwise $a_1\equiv a_2\pmod p$ which is a contradiction as $a_1,a_2,$ are distinct integers of $[1,p]$. Using this with Lemma $10$ implies $n$ is $3$-permissible if and only if $n$ is a power of $2$.

We can also use Lemma $10$ to determine an infinite number of $n$ that are $4$-permissible. The following permutations show that all primes less than or equal to $23$ are $4$-permissible. 
$$\begin{array}{c}

                \hspace{4.5em} 2,1 \hspace{4.5em} \vspace{0.5em} \\

                \hspace{3.5em} 3,1,2 \hspace{3.5em} \vspace{0.5em}\\

                \hspace{3em} 5,1,3,4,2 \hspace{3em} \vspace{0.5em}\\

                \hspace{2em} 7,4,2,6,3,5,1 \hspace{2em} \vspace{0.5em}\\

                \hspace{1em} 11,2,7,9,10,6,1,4,8,5,3 \hspace{1em} \vspace{0.5em}\\

                 \hspace{2em}13,10,3,5,4,12,8,1,2,9,6,7,11  \hspace{2em} \vspace{0.5em} \\

                 \hspace{2em}17,3,8,7,12,1,5,13,4,15,14,9,16,6,11,10,2  \hspace{2em} \vspace{0.5em}\\

                 \hspace{2em}19, 3, 1, 9, 6, 5, 2, 15, 7, 13, 14, 10, 8, 17, 4, 16, 18, 12, 11 \hspace{2em} \vspace{0.5em}\\

                 \hspace{2em}23, 1, 22, 15, 3, 11, 2, 13, 18, 17, 20, 7, 8, 10, 5, 9, 6, 21, 4, 19, 14, 12, 16 \hspace{2em} \\
                
\end{array}$$

So all $n$ with no prime factors greater than $23$ are $4$-permissible. 

We now find the structure of all permutations of $[1,2^k]$ which avoid $3$-APs mod $2^k$. This structure had been somewhat described in \cite{Goh21}.

\begin{theorem} [\cite{Goh21}]
Let $n\geq 1$ be an integer. The number of permutations of
$[1,n]$ that do not contain any $3$-APs mod $n$ is $2^{n-1}$ if $n = 2^k$ for some $k \geq 1$, and is $0$ otherwise. A permutation, $Q$, of $[1,2^k]$ that contains no $3$-APs mod $2^k$ consists of $2^{k-1}$ elements of the same parity, followed by the $2^{k-1}$ elements of the opposite parity.\end{theorem}

They note in the proof that as the subsequence of the $2^{k-1}$ terms of the same parity in $Q$ must avoid $3$-APs, these subsequences must be $2S$ or $2T-1$ for some $S$ or $T$ that are permutations of $[1,2^{k-1}]$ that avoid $3$-APs mod $2^{k-1}$. From this they prove that all permutations of $[1,2^k]$ that avoid $3$-APs mod $2^k$ are exactly those of the form $2S,2T-1$ or $2T-1,2S$ for some $S$ and $T$ which are permutations of $[1,2^{k-1}]$ that avoid $3$-APs mod $2^{k-1}$. We extend this characterisation in more detail. 

Let $[a+mb]'$ be a permutation of all the integers in $[1,2^k]$ that are equivalent to $a+mb \pmod{2^k}$, where $a, b$ are fixed integers and $m$ can be any odd integer, that avoids $3$-APs mod $2^k$. That is, $[a+mb]'$ contains all integers in $[1,2^k]$ that are equivalent to $a+b \pmod{2b}$. 

\begin{theorem}\label{12} Permutations of $[1,2^k]$ that avoid $3$-APs mod $2^k$ are exactly those of the form $$a,a+2^{k-1}\pmod{2^k},[a+m2^{k-2}]',[a+m2^{k-3}]',\ldots,[a+m2]',[a+m]'$$ for an $a\in[1,2^k].$
There are $2^{{2^k}-1}$ such permutations.\end{theorem}
\begin{proof} Let $P$ be a permutation of $[1,2^k]$ that avoids $3$-APs mod $2^k$. We shall prove this has the claimed structure by induction on $r$ where $r$ is defined below. If $k=0$ or $k=1$, then $P$ trivially satisfies the structure in the theorem. Note that this structure can be rewritten as $$[a+m2^k]',[a+m2^{k-1}]',[a+m2^{k-2}]',[a+m2^{k-3}]',\ldots,[a+m2]',[a+m]'$$ as the first two blocks only contain one term each. Additionally, the block $[a+m2^{k-t}]'$ has $2^{t-1}$ terms in it for $1\leq t\leq k$. Let the first term of $P$ be $a$. Then denote the second term by $b$. It is clear that $a,b,2b-a \pmod{2^k}$ is a $3$-AP mod $2^k$ in $P$, unless $2b-a\equiv a \pmod{2^k}$ or $2b-a\equiv b\pmod{2^k}$. If $2b-a\equiv b\pmod{2^k}$ then $a\equiv b\pmod{2^k}$ but as $a,b\in[1,2^k]$, we see that $a=b$. But $a$ and $b$ are distinct. If $2b-a\equiv a \pmod{2^k}$ then $b\equiv a\pmod{2^{k-1}}$ so $b=a$ or $b=a+2^{k-1}\pmod{2^k}$. But $b$ is distinct from $a$ so the second term is $a+2^{k-1}\pmod{2^k}$. So the first $2^1$ terms are in the order as claimed. 

Now assume that the first $2^r$ terms are in the order as claimed for some $1\leq r\leq k-1$. This corresponds to the first $2^r$ terms being $$[a+m2^k]',[a+m2^{k-1}]',\ldots,[a+2^{k-r}]' \text{ as } 1+\sum_{i=1}^{r}2^{i-1}=2^r.$$ Label the $2^r+i$th term in $P$ by $x_i$. Then the following term is $x_1$. For any term $y$ appearing before $x_1$, if $2x_1-y\pmod{2^k}$ appeared in $P$ after $x_1$, then $y,x_1,2x_1-y\pmod{2^k}$ forms a $3$-AP mod $2^k$. We see that $2x_1-y\pmod{2^k}$ cannot be equal to $x_1$ as this implies that $x_1\equiv y\pmod{2^k}$ and therefore that $x_1=y$ which is a contradiction as $x_1$ and $y$ are distinct terms of the permutation of $[1,2^k]$. So for all terms $y$ appearing before $x_1$, we know $2x_1-y\pmod{2^k}$ must appear in $P$ before $x_1$. All terms appearing in $P$ before $x_1$ are equivalent to $a+2^{k-l}\pmod{2^k}$ for $0\leq l\leq r$ so they are equivalent to $a\pmod{2^{k-r}}$. So for any $y$ before $x_1$, it follows that $$2x_1-y\equiv a\pmod{2^{k-r}} \text{ and } y\equiv a\pmod{2^{k-r}}.$$ Therefore, $2x_1\equiv 2a\pmod{2^{k-r}}$ and so $x_1\equiv a\pmod{2^{k-r-1}}$. So $2^{k-(r+1)}$ divides $x_1-a$. Assume $x_1$ is not equivalent to $a+m2^{k-(r+1)}\pmod{2^k}$ for some odd $m$. Then $x_1\equiv a+m2^{k-l}\pmod{2^k}$ for some odd $m$ and some $r+1<l\leq k$ as all terms with $l\leq r$ occur before $x_1$. However, this implies that $2^{k-l}$ is the highest power of $2$ that divides $x_1-a$ and so $2^{k-(r+1)}$ does not divide $x_1-a$. This contradiction means that $x_1\equiv a+m2^{k-(r+1)}\pmod{2^k}$ for some odd $m$.

Now consider $x_i$ for some fixed $i$, for $2\leq i\leq 2^r$. Then for all terms $y$ before $x_1$ in $P$, we know that $2x_i-y\pmod{2^k}$ cannot appear after $x_i$ in $P$, as $P$ avoids $3$-APs mod $2^k$. We see $2x_i-y\pmod{2^k}$ cannot equal $x_i$ as then $x_i\equiv y\pmod{2^k}$ and so $x_i=y$ which is a contradiction as $y$ appears in $P$ before $x_i$. It follows that $2x_i-y\pmod{2^k}$ appears in $P$ before $x_i$. If there exists $y$ before $x_1$ such that $2x_i-y\pmod{2^k}$ occurs before $x_1$ then by the argument above for $x_1$, it is clear that $$x_i\equiv a+m2^{k-(r+1)}\pmod{2^k}$$ for $m$ odd.

Let the first $2^r$ terms of $P$ be $y_1,y_2,\ldots,y_{2^r}$ and assume there is no $y_j$ such that $2x_i-y_j\pmod{2^k}$ appears in $P$ before $x_1$. So for each of the $y_j$ before $x_1$ in $P$, for $1\leq j\leq 2^r$, there exists $n$, for $1\leq n<i\leq 2^r$, such that $2x_i-y_j\equiv x_n\pmod{2^k}$. As there are $2^r$ values of $y_j$ and less than $2^r$ values of $n$, by the pigeonhole principle there must exist distinct $1\leq j<h\leq 2^r$ and an $n$, for $1\leq n<i\leq 2^r$, such that $$2x_i-y_j\equiv x_n\equiv 2x_i-y_h\pmod{2^k}.$$ However, then $y_j\equiv y_h\pmod{2^k}$ but as $y_j$ and $y_h$ are terms of a permutation of $[1,2^k]$, this implies $y_j=y_h$ and so $j=h$. This is a contradiction.

So $x_i\equiv a+m_i2^{k-(r+1)}\pmod{2^k}$ for $m_i$ odd and $1\leq i\leq 2^r$. As there are $2^r$ terms equivalent to $a+m2^{k-(r+1)}\pmod{2^k}$ for $m$ odd in $[1,2^k]$, it is clear that $x_1, x_2,\ldots, x_{2^r}$ are exactly these terms. As $P$ avoids $3$-APs mod $2^k$, so does $x_1,x_2,\ldots,x_{2^r}$. It follows that $x_1,x_2,\ldots,x_{2^r}$ is a permutation of the form $[a+m2^{k-(r+1)}]'$. So the first $2^{r+1}$ terms of $P$ are ordered as claimed.
Therefore the induction is completed. 
So all the permutations of $[1,2^k]$ that avoid $3$-APs mod $2^k$ are of the form $$a,a+2^{k-1}\pmod{2^k},[a+m2^{k-2}]',[a+m2^{k-3}]',\ldots,[a+m2]',[a+m]'.$$

Now we show that all permutations of this form avoid $3$-APs mod $2^k$. Let $a_1,a_2,a_3,$ be a $3$-AP mod $2^k$ in such a permutation. Then $a_i\equiv a+m_i2^{k-l_i}\pmod{2^k}$ for odd $m_i$ and $i\in\{1,2,3\}$. We see that $0\leq l_1\leq l_2\leq l_3\leq k$ as $a_2$ appears after $a_1$ and $a_3$ appears after $a_2$. 

\medskip
\noindent
$\textbf{Case 1}$: If $l_1=l_2=s$ then $$a_3\equiv 2a_2-a_1\equiv a+2^{k-s}(2m_2-m_1)\equiv a+r2^{k-s}\pmod{2^k}$$ for $r=2n-m$ odd. So $a_1,a_2,a_3,$ are all in $[a+m2^{k-s}]'$ which avoids $3$-APs mod $2^k$. \hfill $\diamond$

\medskip
\noindent
$\textbf{Case 2}$: If $l_2>l_1$ then $l_2\geq l_1+1\geq 1$. So $$a_3\equiv 2a_2-a_1\equiv a+2^{k-l_2+1}(m_2-m_12^{l_2-l_1-1})\pmod{2^k}$$ which implies that $2^{k-l_2+1}$ divides $a_3-a$ as $m_2-m_12^{l_2-l_1-1}$ is an integer and $k-l_2+1\leq k$. But $a_3\equiv a+m_32^{k-l_3}\pmod{2^k}$ for $m_3$ odd and $l_2\leq l_3\leq k$. This is a contradiction as $k-l_3\leq k-l_2$ and so the highest power of $2$ that could divide $a_3-a$ is $2^{k-l_2}$. \hfill $\diamond$

\medskip
\noindent
So permutations of $[1,2^k]$ that avoid $3$-APs mod $2^k$ are exactly those of the form $$a,a+2^{k-1}\pmod{2^k},[a+m2^{k-2}]',[a+m2^{k-3}]',\ldots,[a+m2]',[a+m]'.$$

 Let $1\leq t\leq k$. Then a permutation of $[1,2^{t-1}]$ that avoids $3$-APs mod $2^{t-1}$, has been multiplied by $2^{k-t+1}$, had $a-2^{k-t}$ added to its terms and then been reduced mod $2^k$ is a permutation of the form $[a+m2^{k-t}]'$. Additionally each permutation of the form $[a+m2^{k-t}]'$ that has had $2^{k-t}-a$ added to its terms, been divided by $2^{k-t+1}$ and reduced mod $2^{t-1}$ is a permutation of $[1,2^{t-1}]$ that avoids $3$-APs mod $2^{t-1}$.  
 So the number of possible permutations that are of the form $[a+m2^{k-t}]'$ is the same as the number of permutations of $[1,2^{t-1}]$ that avoid $3$-APs mod $2^{t-1}$. 
 
 Let $\phi(k)$ be the number of permutations of $[1,2^k]$ that avoid $3$-APs mod $2^k$. Then as this is the number of permutations of the form $$a,a+2^{k-1}\pmod{2^k},[a+m2^{k-2}]',[a+m2^{k-3}]',\ldots,[a+m2]',[a+m]',$$ we see that $$\phi(k)=2^k\phi(k-1)\phi(k-2)\ldots\phi(2)\phi(1)\phi(0)$$ as there are $2^k$ choices for $a$ and there are $\phi(t-1)$ choices for $[a+m2^{k-t}]'$ with $1\leq t\leq k$. As $\phi(0)=1$ is clear, it is straightforward to check that $\phi(k)=2^{2^k-1}$. \end{proof}

 \section{Open Questions}

 There are still many unknown questions in this area. Here we lay out some of these, including summarising those we have mentioned.

 \begin{question}[\cite{Davis77,MR0592420,LeSaulnier11}]
 Do there exist permutations of the positive integers or the integers that avoid $4$-APs?
 \end{question}

This is the most obvious open question and was asked in \cite{Davis77} in 1977, the first paper on the topic. Showing $\beta_{\mathbb{Z}^+}(4)=1$ is a step towards resolving this in the positive integer case, but it seems some new techniques will need to be applied to answer it, either in showing that all permutations of the (positive) integers contain $4$-APs or in constructing a permutation of the (positive) integers that avoids them.   

 \begin{question}[\cite{Davis77,MR0592420}]
 Do there exists doubly infinite permutations of the integers that avoid $4$-APs?
\end{question}
It was shown in \cite{Davis77} that there exist doubly infinite permutations of the positive integers that avoid $4$-APs, but it is not known if the same is true of the integers. Note that if there exists a permutation of the integers, $P$, that avoids $4$-APs then there exists a doubly infinite permutation of the integers, $2P-1$ going to the left and $2P$ going to the right, that avoids $4$-APs.  

 \begin{question}[\cite{Davis77,Geneson18,LeSaulnier11}]
 What are the correct values of $\alpha_{\mathbb{Z}}(3),\alpha_{\mathbb{Z}^+}(3), \beta_{\mathbb{Z}}(3), \beta_{\mathbb{Z}^+}(3)$ and $\beta_{\mathbb{Z}}(4)$?
 \end{question}

As it seems likely that the density functions for subsets that avoid $3$-APs are significantly less than $1$, this would require a way to prove a nontrivial upper bound, (less than $1$), on these functions which has not been achieved yet. As mentioned before, it was conjectured in \cite{LeSaulnier11} that $\alpha_{\mathbb{Z}^+}(3)=\frac{1}{2}$ and $\beta_{\mathbb{Z}^+}(3)=\frac{1}{4}$, so progress on this front would be of interest.

\begin{question}[\cite{Davis77,Karolyi17}]
    Are there permutations of $[1,p]$ that avoid $4$-APs mod $p$ for all primes $p$?
\end{question}

This would answer the question for all $n$ as explained above. It is known that all that, for all $n$ and $k\geq 5$, there are permutations of $[1,n]$ that avoid $k$-APs, as proved in \cite{Karolyi17}. A further question would be to investigate the number of permutations that avoid $k$-APs mod $n$ for $k>3$, as suggested in \cite{Goh21}.

\begin{question}[\cite{Goh21}]
How many permutations of $[1,n]$ avoid $k$-APs for $k\geq 3$?
\end{question}

This question has been studied for $k=3$ in \cite{Correll2017ANO,Davis77,LeSaulnier11,sharma2009enumerating}, and lower bounds for larger $k$ have been given in \cite{Geneson18}.  

\begin{question}[\cite{Davis77,MR0592420,Geneson18,LeSaulnier11}]
Is it possible to partition the positive integers into two sets, each of which can be permuted to avoid $3$-APs?
\end{question}

In \cite{Davis77}, the positive integers were partitioned into three sets, each of which could be permuted to avoid $3$-APs. In \cite{MR0592420}, it was mentioned that the corresponding question for integers was open. We note that the integers can be partitioned into $3$ sets each of which can be permuted to avoid $3$-APs. 
Let $K$ be as in Theorem \ref{8}. Define $X_i'$ as a permutation of $[\lceil \frac{8}{5}2^i\rceil, \lfloor \frac{64}{25}2^i\rfloor]$ that avoids $3$-APs and $Y_i'$ as a permutation of $[-\lfloor \frac{64}{25}2^i\rfloor, -\lceil \frac{8}{5}2^i\rceil]$ that avoids $3$-APs. Then let $$K'= 0,1,Y_0',X_1', Y_2', X_3',Y_4',\ldots,X_{2k-1}',Y_{2k}'\ldots$$ Similarly, define $X_i''$ as a permutation of $[\lceil \frac{64}{25}2^i\rceil, 2^{i+2})$ that avoids $3$-APs and $Y_i''$ as a permutation of $(-2^{i+2}, -\lceil \frac{64}{25} 2^i\rceil]$ that avoids $3$-APs. Then let $$K''= Y_0'', X_1'', Y_2'', X_3'',Y_4'',\ldots,X_{2k-1}'',Y_{2k}''\ldots$$  
As $K'$ and $K''$ are very similar to $K$, (the endpoints of their intervals are scaled by $\frac{8}{5}$ and $\frac{64}{25}$ respectively compared to $K$), it can be shown by the same argument as in Theorem \ref{8} that they both avoid $3$-APs. As the set $K\cup K'\cup K''$ is the set of all integers and each avoids $3$-APs, we can partition the integers into $3$ sequences that each avoid $3$-APs. It was noted in \cite{LeSaulnier11} that if $\alpha_{\mathbb{Z^+}}(3)+\beta_{\mathbb{Z^+}}(3)<1$ then the positive integers cannot be partitioned into $2$ sequences that each avoid $3$-APs. Similarly, if $\alpha_{\mathbb{Z}}(3)+\beta_{\mathbb{Z}}(3)<1$ then the integers cannot be partitioned into $2$ sequences that each avoid $3$-APs. 

\bibliographystyle{plain}
\bibliography{Latex}

\end{document}